\documentclass[10pt]{amsart}
\usepackage{amsmath,amssymb,amsthm,mathtools,mathrsfs,booktabs,array,longtable}
\usepackage[T1]{fontenc}
\usepackage{lmodern}
\usepackage{needspace}
\usepackage{microtype}
\usepackage[colorlinks=false,linkcolor=black,citecolor=black,urlcolor=black,
  linkbordercolor={1 0 0},citebordercolor={1 0 0},urlbordercolor={1 0 0},
  pdfborder={0 0 1}]{hyperref}
\hypersetup{pdftitle={Consecutive Rankin-Cohen Bases, Full-Spark Periods, and Divisor-Tau Congruences},pdfauthor={Kelvin Lam}}
\newtheorem{theorem}{Theorem}[section]
\newtheorem{proposition}[theorem]{Proposition}
\newtheorem{lemma}[theorem]{Lemma}
\newtheorem{corollary}[theorem]{Corollary}
\theoremstyle{definition}
\newtheorem{definition}[theorem]{Definition}
\theoremstyle{remark}
\newtheorem{remark}[theorem]{Remark}
\numberwithin{equation}{section}

\newcommand{\sgn}{\operatorname{sgn}}
\newcommand{\diag}{\operatorname{diag}}
\newcommand{\SL}{\mathrm{SL}_2(\mathbb Z)}
\newcommand{\RR}{\mathbb R}

\newcommand{\CC}{\mathbb C}
\newcommand{\HH}{\mathbb H}
\newcommand{\Tcal}{\mathcal T}

\title[FULL-SPARK PERIODS AND CONSECUTIVE BRACKETS]{CONSECUTIVE RANKIN--COHEN BASES, FULL-SPARK PERIODS, AND DIVISOR--TAU CONGRUENCES}
\author{Kelvin Lam}
\date{}

\newcommand{\Q}{\mathbf Q}

\allowdisplaybreaks
\begin{document}
\begin{abstract}
For $r\ge1$, let $K=2r+14$ and $d=\dim S_K$. Our analytic starting point is a
strict positivity theorem for determinants of Mellin period functionals. It implies
a full-spark statement for periods in a fundamental half-range: every collection
of at most $d$ distinct period functionals there is linearly independent, with no
parity restriction. In particular this proves the mixed odd--even period
independence conjecture of Xue.

As a consequence, the consecutive first Rankin--Cohen brackets
$[E_{2r+10-2j},E_{2j+2}]_1$, $1\le j\le d$, form a basis of $S_K$; more
generally we determine the sign of every admissible ordered determinant of first
Eisenstein brackets. Applying the consecutive basis to $E_{2r+2}\Delta$ gives a
uniform exact all-weight formula for
\[
C_r(n)=\sum_{m=1}^{n-1}\sigma_{2r+1}(m)\tau(n-m).
\]
We then reduce this same characteristic-zero coordinate identity modulo primes.
For every all-weight good prime $p$, the reduced canonical blocks form a basis of
the reduction of the natural $\mathbf Z_{(p)}$ cusp-form lattice, and the resulting
coefficient vector is uniquely determined by
$D_j\equiv Z_j(r)\pmod p$, where $Z_j(r)$ is the characteristic-zero Cramer
vector. Every prime $p\ge K-3$ lies in the coefficient-integral range, and for
each fixed $r$ only finitely many such primes fail the determinant condition.
Thus canonical prime-wise reductions exist for infinitely many primes in every
weight. Their genuinely sparse Ramanujan-type specializations occur when Cramer
coordinates vanish modulo $p$; we give exact numerator criteria and simultaneous
vanishing invariants for this phenomenon.

Finally, after division by $\Delta$, the canonical blocks lie in $M_{2r+2}$, and
for homogeneous modular forms $f,g\in\mathbf Q[E_4,E_6]$ we prove the Jacobian
identity
\[
\frac{[f,g]_1}{\Delta}
=-3456\,\det\begin{pmatrix}f_{E_4}&f_{E_6}\\ g_{E_4}&g_{E_6}\end{pmatrix}.
\]
This converts the Cramer system into a one-variable coordinate problem in
$T=E_6^2/E_4^3$.
\end{abstract}
\maketitle

\section{Introduction}
For positive integers $r,n$, put
\begin{equation}\label{eq:C}
C_r(n)=\sum_{m=1}^{n-1}\sigma_{2r+1}(m)\tau(n-m),
\qquad \sigma_a(n)=\sum_{t\mid n}t^a,
\end{equation}
where $\tau(n)$ is Ramanujan's tau function and an empty sum is zero.
The generating function of this convolution is encoded by the cusp form
$E_{2r+2}\Delta$. Consequently a suitable basis of $S_{2r+14}$ converts
\eqref{eq:C} into identities involving only divisor sums and the single
term $\tau(n)$ at the target index. Rankin--Cohen brackets have long been
used in such identities; see, for example, \cite{RS}. The additional
issue addressed here is the independence of a particular minimal family,
with its orientation fixed by the Fourier expansion.

The basis step is analytic in every dimension, but the Mellin determinant used to
prove it has a stronger consequence. It yields a full-spark theorem for periods:
within a fundamental half-range, any collection of at most $d=\dim S_K$ distinct
period functionals is linearly independent, even when odd and even indices are
mixed. Theorem~\ref{thm:intro-periods}, proved as Corollary~\ref{cor:full-spark-periods},
makes this precise, and Corollary~\ref{cor:xue-mixed} identifies the mixed-parity conjecture of
Xue~\cite[Conjecture~1.2]{XueMixed} as an immediate consequence.

The arithmetic statement obtained from the consecutive bracket basis is a direct
modular reduction of the \emph{same} all-weight coordinate formula. Thus no new
family of modular forms is introduced in characteristic $p$: the canonical atoms
are precisely the consecutive normalized Rankin--Cohen blocks already used in
characteristic zero. Section~\ref{sec:uniquecong} works throughout with the
reduction of a specified $\mathbf Z_{(p)}$ lattice, rather than with an
unspecified full space of modular forms in characteristic $p$. At every all-weight
good prime the reduced canonical blocks remain a basis of this lattice reduction,
and the resulting congruence has a unique coefficient vector, equal to the
reduction of the characteristic-zero Cramer coordinates $Z_j(r)$. Since every
prime $p\ge K-3$ is automatically in the coefficient-integral range and the
nonzero rational determinant $\det N_r$ has only finitely many prime divisors,
every fixed $r$ has infinitely many good primes.

The word \emph{canonical} is important here. The existence of these reductions
for almost every prime is a basis-theoretic consequence of the characteristic-zero
identity; by itself it is not an exceptional Ramanujan congruence. The additional
arithmetic phenomenon occurs when one or more Cramer coordinates vanish modulo a
good prime, forcing canonical blocks to disappear and producing genuinely sparse
formulas. We reserve the phrase \emph{Ramanujan-type} for these sparse
specializations. Uniqueness is always relative to the fixed consecutive
Rankin--Cohen basis supplied by Theorem~\ref{thm:mainbasis}; it is not a claim
that no other identity can be written using a different collection of modular
forms. Section~\ref{sec:cramer} records the exact Cramer arithmetic of coefficient
vanishing, Section~\ref{sec:newcongruences} gives explicit sparse congruences, and
Section~\ref{sec:jacobian} identifies the quotient by $\Delta$ with a Jacobian
basis in $\mathbf Q[E_4,E_6]$ and reduces the coordinate system to one variable.

Throughout, the bracket convention is
\begin{equation}\label{eq:RCintro}
[f,g]_1=kfDg-\ell(Df)g,\qquad D=q\frac{d}{dq},
\quad f\in M_k,\quad g\in M_\ell.
\end{equation}
Changing this convention to its negative changes a $d$-column determinant
by $(-1)^d$; the convention is therefore essential to the sign theorem.

\begin{theorem}[Full-spark period theorem]\label{thm:intro-periods}
Let $K\ge12$ be even, $d=\dim S_K>0$, and $h=(K-2)/2$. If
\[
0\le u_1<\cdots<u_m<h,\qquad m\le d,
\]
then the period functionals $\rho_{u_1},\ldots,\rho_{u_m}$ are linearly
independent in $S_K^*$. If $K\equiv0\pmod4$, the endpoint $u_m=h$ may also be
included. Thus the assertion applies to arbitrary mixtures of odd and even
periods in the fundamental half-range.
\end{theorem}

\Needspace{13\baselineskip}
\begin{theorem}[Consecutive basis and determinant sign]\label{thm:mainbasis}
Let $r\ge1$, $K=2r+14$, and $d=\dim S_K=\dim M_{2r+2}$. Set
\begin{equation}\label{eq:ab}
a_j=2r+10-2j,\qquad b_j=2j+2,\qquad
R_j=[E_{a_j},E_{b_j}]_1\quad(1\le j\le d).
\end{equation}
Then $R_1,\ldots,R_d$ form a basis of $S_K$. With rows and columns in
increasing order, the matrix
\begin{equation}\label{eq:M}
M_r=\bigl([q^n]R_j\bigr)_{1\le n,j\le d}
\end{equation}
satisfies
\begin{equation}\label{eq:mainsign}
(-1)^{\binom d2}\det M_r>0.
\end{equation}
\end{theorem}

Here and below, $E_k$ has constant term $1$. For an explicit statement of
the convolution consequence, define, for every $n\ge1$,
\begin{align}
h_{r,j}(n)={}&-\frac{B_{a_j}}2\,n\sigma_{b_j-1}(n)
 +\frac{B_{b_j}}2\,n\sigma_{a_j-1}(n)\notag\\
&+\sum_{m=1}^{n-1}\bigl(a_jn-(a_j+b_j)m\bigr)
\sigma_{a_j-1}(m)\sigma_{b_j-1}(n-m),\label{eq:h}
\end{align}
and put
\begin{equation}\label{eq:Nz}
N_r=(h_{r,j}(n))_{1\le n,j\le d},\qquad
z_r=B_{2r+2}\bigl([q^n](E_{2r+2}\Delta)\bigr)_{n=1}^d.
\end{equation}
Write $N_r[j\leftarrow z_r]$ for replacement of column $j$ by $z_r$.

\begin{theorem}[All-weight convolution formula]\label{thm:mainconv}
For every $r\ge1$, $N_r$ is nonsingular and the rational numbers
\begin{equation}\label{eq:Zcramer}
Z_j(r)=\frac{\det N_r[j\leftarrow z_r]}{\det N_r},\qquad 1\le j\le d,
\end{equation}
are well defined independently of $n$. For every $n\ge1$,
\begin{equation}\label{eq:compactC}
C_r(n)=\frac{1}{4(r+1)}
\left(B_{2r+2}\tau(n)-\sum_{j=1}^dZ_j(r)h_{r,j}(n)\right).
\end{equation}
Equivalently, if $h_r(n)=(h_{r,1}(n),\ldots,h_{r,d}(n))^{\mathsf T}$, then
\begin{equation}\label{eq:border}
C_r(n)=\frac{1}{4(r+1)\det N_r}
\det\begin{pmatrix}N_r&z_r\\h_r(n)^{\mathsf T}&B_{2r+2}\tau(n)\end{pmatrix}.
\end{equation}
\end{theorem}

Thus the result gives a finite exact determinantal formula in every
weight. It does not assert a bounded-length scalar expression for all coordinate
coefficients as $r$ grows: the size $d$ of the coordinate problem grows with the
weight. Theorem~\ref{thm:mainbasis} proves directly that the required denominators
are nonzero in every dimension.

For later use write
\[
G_k=-\frac{B_k}{2k}E_k,\qquad
H_j=[G_{a_j},G_{b_j}]_1,\qquad
T_r=-2(2r+2)G_{2r+2}\Delta.
\]
Then $[q^n]H_j=h_{r,j}(n)$ and Theorem~\ref{thm:mainconv} is equivalent to
\[
T_r=\sum_{j=1}^d Z_j(r)H_j.
\]

\begin{theorem}[Unique canonical all-weight congruences]\label{thm:maincongruence}
Let $K=2r+14$, $s=2r+2$, and $d=\dim S_K$. Let $p$ be a prime for which the
canonical blocks and target are $p$-integral, $2s$ is a $p$-adic unit, and
$\det N_r$ is a $p$-adic unit. Then there is a unique vector
$(D_1,\ldots,D_d)\in\mathbf F_p^d$ such that, for every $n\ge1$,
\[
C_r(n)\equiv
\frac{B_s}{2s}\tau(n)-\frac1{2s}\sum_{j=1}^dD_jh_{r,j}(n)\pmod p.
\]
It is given by
\[
D_j\equiv Z_j(r)\pmod p\qquad(1\le j\le d).
\]
If $Z_j(r)=U_j(r)/V_j(r)$ is in lowest terms, then $V_j(r)$ is a $p$-adic unit
and
\[
D_j=0\quad\Longleftrightarrow\quad p\mid U_j(r).
\]
Moreover every prime $p\ge K-3$ is automatically in the coefficient-integral
range. For each fixed $r$, all but finitely many primes in this range satisfy the
determinant condition, so every $C_r$ has infinitely many prime-wise unique
canonical all-weight congruences.
\end{theorem}

\subsection{Relation to period results}
Lei--Ni--Xue \cite{LNX} proved strong finite and asymptotic independence results
for even periods and related first Rankin--Cohen brackets to them. Xue's later
work on odd--even interaction proved pairwise mixed-parity independence in a
large range and formulated the conjecture that any mixed collection in a
fundamental half-range should be independent whenever its cardinality does not
exceed $\dim S_K$ \cite[Conjecture~1.2]{XueMixed}. Ni--Song--Su--Xue--Yin
\cite{NSSXY} subsequently obtained explicit spanning subsets and recorded the
corresponding arbitrary-family even-period independence problem as
\cite[Conjecture~4.2]{NSSXY}.

Theorem~\ref{thm:perioddet} gives more than is needed for the bracket basis: its
strict Mellin determinant is insensitive to the parity of the integral period
indices. After reflection across the center by the functional equation, it proves
that every subset of at most $d$ period functionals in a fundamental half-range
is independent; see Corollaries~\ref{cor:full-spark-periods} and
\ref{cor:xue-mixed}. Thus the even-period arbitrary-family assertion and Xue's
mixed odd--even conjecture are both consequences of the same total-positivity
argument. Fukuhara's explicit basis \cite[Theorem 2.2]{Fukuhara} gives a
particular basis of odd periods; the present statement is of a different type,
namely a full-spark assertion for all parities in the half-range.

For the determinant orientation of first brackets, Section~5 translates the
published period pairing \cite[Proposition 1.3]{LNX} into a real Mellin pairing
with an explicitly signed column factor. The classical pairing identity is a
bridge used here, not a new result.

\section{Modular forms, periods, and coefficient extraction}
All modular forms are for $\SL$, and $q=e^{2\pi iz}$ for $z\in\HH$.
We write $S_K(\RR)$ for the real vector space of forms with real Fourier
coefficients, and define $M_s(\RR)$ analogously. We also write $S_K(\Q)$
for the $\Q$-vector space of cusp forms with rational Fourier coefficients.
Unless a field is indicated, dimensions of modular-form spaces are taken over $\CC$.
Our Bernoulli convention is
\[
\frac{t}{e^t-1}=\sum_{m\ge0}B_m\frac{t^m}{m!},\qquad B_1=-\frac12.
\]
For even $k\ge4$,
\begin{equation}\label{eq:E}
E_k=1+e_k\sum_{n\ge1}\sigma_{k-1}(n)q^n,
\qquad e_k=-\frac{2k}{B_k}.
\end{equation}
Also
\begin{equation}\label{eq:Delta}
\Delta=\frac{E_4^3-E_6^2}{1728}
=q\prod_{n\ge1}(1-q^n)^{24}=\sum_{n\ge1}\tau(n)q^n.
\end{equation}
We use the Petersson product
\[
\langle f,g\rangle=\int_{\SL\backslash\HH}
 f(z)\overline{g(z)}y^K\,\frac{dx\,dy}{y^2},\qquad z=x+iy,
\]
linear in the first variable. On $S_K(\RR)$ it is real, symmetric and
positive definite, as is also clear from a real orthogonal Hecke basis.

\begin{lemma}\label{lem:dim}
For even $s\ge4$, multiplication by $\Delta$ is an isomorphism
$M_s\longrightarrow S_{s+12}$ over $\CC$, and restricts to an
isomorphism $M_s(\RR)\longrightarrow S_{s+12}(\RR)$ over $\RR$.
Moreover,
\[
S_K(\RR)\otimes_{\RR}\CC\simeq S_K,
\qquad \dim_{\RR}S_K(\RR)=\dim_{\CC}S_K.
\]
Furthermore,
\[
\dim M_s=\begin{cases}
\lfloor s/12\rfloor,&s\equiv2\pmod{12},\\
\lfloor s/12\rfloor+1,&s\not\equiv2\pmod{12}.
\end{cases}
\]
If $K\ge12$ is even and $d=\dim S_K>0$, then uniquely
\begin{equation}\label{eq:kappa}
K=12d+\kappa,\qquad \kappa\in\{0,4,6,8,10,14\}.
\end{equation}
\end{lemma}
\begin{proof}
The discriminant has no zero on $\HH$ and has a simple zero at the cusp,
so division by $\Delta$ proves the complex isomorphism. Both
multiplication and division preserve real Fourier coefficients, which
also proves the real isomorphism. The monomial bases supplied by
$\CC[E_4,E_6]$ and $S_K=\Delta M_{K-12}$ have real Fourier
coefficients. Coefficientwise complex conjugation therefore fixes
exactly their real spans. This proves the complexification statement
and the equality of real and complex dimensions. The dimension formula
follows from the standard graded-ring description
$\bigoplus M_s=\CC[E_4,E_6]$; see \cite{Zagier}. Applying it to
$S_K=\Delta M_{K-12}$ gives \eqref{eq:kappa}. In particular, residue
class $K\equiv2\pmod{12}$ corresponds to $\kappa=14$, not to $2$.
\end{proof}

\begin{proposition}\label{prop:extract}
Let $F_r=E_{2r+2}\Delta$. For every $n\ge1$,
\begin{equation}\label{eq:extract}
[q^n]F_r=\tau(n)-\frac{4(r+1)}{B_{2r+2}}C_r(n).
\end{equation}
\end{proposition}
\begin{proof}
Multiply the two series in \eqref{eq:E} and \eqref{eq:Delta}; the constant
term of $E_{2r+2}$ gives $\tau(n)$ and its remaining terms give the
convolution in \eqref{eq:C}.
\end{proof}

For real $t$ define
\begin{equation}\label{eq:It}
I_t(f)=\int_0^\infty f(iy)y^t\,dy.
\end{equation}
The cusp condition and modular inversion give convergence for every
real $t$. For a nonnegative integer $t$,
\begin{equation}\label{eq:Mellin}
I_t(f)=\frac{\Gamma(t+1)}{(2\pi)^{t+1}}L(f,t+1),\qquad
\rho_t(f):=\int_0^{i\infty}f(z)z^t\,dz=i^{t+1}I_t(f).
\end{equation}
The $L$-value is understood by analytic continuation where necessary.

\section{Strictly positive determinants of hyperbolic kernels}
We give the elementary determinant facts needed later, including the
strictness that is essential for nonvanishing.

\begin{lemma}[Generalized Vandermonde positivity]\label{lem:V}
Let $0<x_1<\cdots<x_d$ and let $0\le\nu_1<\cdots<\nu_d$ be integers.
Then $\det(x_i^{\nu_j})_{i,j=1}^d>0$.
\end{lemma}
\begin{proof}
We first prove by induction on the number $m$ of nonzero terms that
\[
F(x)=\sum_{j=1}^{m}c_jx^{\mu_j},\qquad
0\le\mu_1<\cdots<\mu_m,\qquad c_j\in\RR\setminus\{0\},
\]
has at most $m-1$ distinct positive zeros. For $m=1$ it has none.
For $m\ge2$, division by the nowhere-zero function $x^{\mu_1}$
on $(0,\infty)$ preserves its zeros and gives
\[
G(x)=c_1+\sum_{j=2}^{m}c_jx^{\mu_j-\mu_1},\qquad
G'(x)=\sum_{j=2}^{m}c_j(\mu_j-\mu_1)x^{\mu_j-\mu_1-1}.
\]
The derivative has exactly $m-1$ nonzero terms, with distinct
nonnegative integer exponents. By induction it has at most $m-2$
distinct positive zeros. If $G$ has $N\ge1$ distinct positive zeros,
Rolle's theorem, applied between consecutive zeros, gives at least
$N-1$ distinct positive zeros of $G'$. Hence
\[
N-1\le m-2,\qquad N\le m-1.
\]
The case $N=0$ is immediate. The same argument applied to any finite
selection of zeros rules out infinitely many positive zeros. Thus a
nonzero combination of the original $d$ monomials, after discarding
zero coefficients, has at most $m-1\le d-1$ distinct positive zeros.
Only distinct zeros are counted; no assumption of simplicity is made.

If the displayed evaluation matrix were singular, a nonzero real
linear combination of its columns would vanish at all $d$ distinct
points $x_i$, contradicting this bound. Its determinant is therefore
nonzero on the connected region $0<x_1<\cdots<x_d$, so its sign is
constant there.

To determine it, let $x_i=x+\eta u_i$, where $x>0$,
$u_1<\cdots<u_d$ and $\eta\downarrow0$. Taylor expansion shows that the
leading determinant term is
\[
\frac{\eta^{\binom d2}\prod_{p<q}(u_q-u_p)}{\prod_{k=0}^{d-1}k!}
\det\left(\frac{d^{i-1}}{dx^{i-1}}x^{\nu_j}\right)_{i,j=1}^d.
\]
The Wronskian equals
\[
x^{\sum_j\nu_j-\binom d2}\prod_{p<q}(\nu_q-\nu_p)>0.
\]
This determines the sign and proves the lemma.
\end{proof}

\begin{lemma}[Hyperbolic kernels]\label{lem:hyperbolic}
For $0<x_1<\cdots<x_d$ and $0<u_1<\cdots<u_d$,
\begin{equation}\label{eq:hyperbolic}
\det(\cosh(x_i u_j))_{i,j}>0,\qquad
\det(\sinh(x_i u_j))_{i,j}>0.
\end{equation}
For the $\cosh$ determinant one may also allow $x_1=0$ (or, by
symmetry, $u_1=0$), with the other list strictly positive.
Reversing either ordered list changes either determinant by
$(-1)^{\binom d2}$.
\end{lemma}
\begin{proof}
For the finite sum $K_N(x,u)=\sum_{n=0}^Nx^{2n}u^{2n}/(2n)!$,
the Cauchy--Binet identity gives
\begin{equation}\label{eq:CBseries}
\det(K_N(x_i,u_j))=
\sum_{0\le n_1<\cdots<n_d\le N}
\frac{\det(x_i^{2n_\ell})_{i,\ell}\det(u_j^{2n_\ell})_{j,\ell}}
{\prod_{\ell=1}^d(2n_\ell)!}.
\end{equation}
Each term is positive by Lemma~\ref{lem:V}. For $N\ge d-1$, the fixed
term $n_\ell=\ell-1$ gives a positive lower bound independent of $N$.
Entrywise convergence and continuity of determinants let $N\to\infty$,
proving strict positivity for $\cosh$. If $x_1=0$, interpret $x^0=1$
also at zero. Every term in \eqref{eq:CBseries} is nonnegative by continuity.
The fixed term $n_\ell=\ell-1$ is strictly positive: its first determinant
is the ordinary Vandermonde in $x_i^2$, with one node zero, and its second
is a positive Vandermonde in $u_j^2$. Thus the same positive lower bound
works at the endpoint. The same argument with powers
$2n+1$ and factorials $(2n+1)!$ proves the assertion for $\sinh$.
Reversing a list is a permutation of sign $(-1)^{\binom d2}$.
\end{proof}

\section{Positive Mellin determinants and full-spark periods}
Fix an even $K\ge12$ with $d=\dim S_K>0$ and use \eqref{eq:kappa}.
Choose $(\alpha,\beta)$ according to
\begin{center}
\begin{tabular}{c*{6}{c}}
\toprule
$\kappa$&0&4&6&8&10&14\\
$(\alpha,\beta)$&(0,0)&(1,0)&(0,1)&(2,0)&(1,1)&(2,1)\\
\bottomrule
\end{tabular}
\end{center}
Define
\begin{equation}\label{eq:Fbasis}
\Phi_m=\Delta^mE_4^{3(d-m)+\alpha}E_6^\beta,\qquad 1\le m\le d.
\end{equation}

\begin{lemma}\label{lem:eval}
The forms $\Phi_m$ are a basis of $S_K(\RR)$ and satisfy
$\Phi_m=q^m+O(q^{m+1})$. If $1<y_1<\cdots<y_d$, then
\begin{equation}\label{eq:evalsign}
(-1)^{\binom d2}\det(\Phi_m(iy_\nu))_{m,\nu}>0.
\end{equation}
\end{lemma}
\begin{proof}
The weights in \eqref{eq:Fbasis} add up to $12d+4\alpha+6\beta=K$.
The leading coefficients give independence over $\RR$. Since
Lemma~\ref{lem:dim} gives $\dim_{\RR}S_K(\RR)=d$, these $d$ forms
are a real basis. For $y>1$ set
\[
J(y)=\frac{E_4(iy)^3}{\Delta(iy)},\qquad
W(y)=\Delta(iy)^dE_4(iy)^\alpha E_6(iy)^\beta.
\]
Then $\Phi_m(iy)=W(y)J(y)^{d-m}$. We have $E_4(iy)>0$ from its
Fourier series and $\Delta(iy)>0$ from its product. Modularity gives
$E_6(i)=0$, while
\[
\frac{d}{dy}E_6(iy)=1008\pi\sum_{n\ge1}n\sigma_5(n)e^{-2\pi ny}>0.
\]
Thus $E_6(iy)>0$ for $y>1$ and $W(y)>0$. The differential identity
$Dj=-E_4^2E_6/\Delta$ gives
\[
J'(y)=2\pi\frac{E_4(iy)^2E_6(iy)}{\Delta(iy)}>0.
\]
For completeness, the identity follows by differentiating $j=E_4^3/\Delta$
and using $DE_4=(E_2E_4-E_6)/3$ and $D\Delta=E_2\Delta$.
The ordinary Vandermonde determinant now gives
\begin{align*}
\det(\Phi_m(iy_\nu))_{m,\nu}
={}&(-1)^{\binom d2}\left(\prod_{\nu=1}^dW(y_\nu)\right)
\prod_{\mu<\nu}(J(y_\nu)-J(y_\mu)),
\end{align*}
proving \eqref{eq:evalsign}.
\end{proof}

\begin{theorem}[Positive Mellin determinant]\label{thm:perioddet}
If $t_1>\cdots>t_d>(K-2)/2$ are real and
\[
A=(I_{t_j}(\Phi_m))_{m,j=1}^d,
\]
then $\det A>0$. If $K\equiv0\pmod4$, the endpoint $t_d=(K-2)/2$
is allowed as well. In particular these $d$ Mellin functionals form a basis
of the real dual space, and their complex-linear extensions form a basis
of $S_K^*$.
\end{theorem}
\begin{proof}
Put $\varepsilon=(-1)^{K/2}$. Modular inversion gives
$f(i/y)=\varepsilon y^Kf(iy)$, and a change of variable on $(0,1)$ yields
\begin{equation}\label{eq:fold}
I_t(f)=\int_1^\infty f(iy)
\bigl(y^t+\varepsilon y^{K-2-t}\bigr)\,dy.
\end{equation}
Write $h=(K-2)/2$, $\lambda_j=t_j-h$, and
\begin{equation}\label{eq:H}
H_j(y)=2y^h\begin{cases}
\cosh(\lambda_j\log y),&\varepsilon=1,\\
\sinh(\lambda_j\log y),&\varepsilon=-1.
\end{cases}
\end{equation}
Then $A_{mj}=\int_1^\infty\Phi_m(iy)H_j(y)\,dy$.
Here $\lambda_1>\cdots>\lambda_d>0$, except that $\lambda_d=0$ is
allowed when $\varepsilon=1$. Lemma~\ref{lem:hyperbolic}, including its
$\cosh$ endpoint case, gives
\begin{equation}\label{eq:Hsign}
(-1)^{\binom d2}\det(H_j(y_\nu))_{\nu,j}>0
\quad(1<y_1<\cdots<y_d).
\end{equation}

Expanding both determinants and applying Fubini gives the integral
Cauchy--Binet identity
\begin{equation}\label{eq:Andreief}
\det A=\int_{1<y_1<\cdots<y_d<\infty}
\det(\Phi_m(iy_\nu))_{m,\nu}\,
\det(H_j(y_\nu))_{\nu,j}\,d\mathbf y.
\end{equation}
To check the factor, the integral over the entire product domain equals
$d!\det A$. Its integrand is symmetric in the $y_\nu$, so restriction
to the ordered domain cancels the $d!$. Every term is integrable: the
forms decay exponentially at infinity and the kernels have at most
power growth, and all functions are bounded near $y=1$.

By \eqref{eq:evalsign} and \eqref{eq:Hsign} the integrand in
\eqref{eq:Andreief} is strictly positive on its nonempty open domain.
Consequently $\det A>0$, as asserted.
\end{proof}

\begin{corollary}[Full-spark periods in a half-range]
\label{cor:full-spark-periods}
Put $h=(K-2)/2$. Let
\[
0\le u_1<\cdots<u_m<h,
\qquad 1\le m\le d.
\]
Then $\rho_{u_1},\ldots,\rho_{u_m}$ are linearly independent in $S_K^*$.
If $K\equiv0\pmod4$, the endpoint $u_m=h$ may also be included. In
particular, the assertion has no parity restriction on the indices.
\end{corollary}
\begin{proof}
Modular inversion gives
\[
I_u=\varepsilon I_{K-2-u},\qquad \varepsilon=(-1)^{K/2},
\]
and hence, for integral $u$,
\begin{equation}\label{eq:period-reflection}
\rho_{K-2-u}=(-1)^{u+1}\rho_u.
\end{equation}
For $u<h$, the complementary index $t=K-2-u$ is strictly larger than $h$;
if $u=h$, then $t=h$ and this endpoint can occur here only in the case
$K\equiv0\pmod4$, exactly the endpoint case allowed in
Theorem~\ref{thm:perioddet}. If $m<d$, adjoin $d-m$ further distinct real
parameters larger than $h$. Theorem~\ref{thm:perioddet} makes the resulting
$d$ Mellin functionals independent, so the chosen $m$ are independent as
well. Equation~\eqref{eq:Mellin} and the nonzero reflection factors in
\eqref{eq:period-reflection} transfer this independence to the periods.
\end{proof}

\begin{corollary}[Mixed odd--even period conjecture]
\label{cor:xue-mixed}
Let
\[
2\le \ell_1<\cdots<\ell_a<K/2-1,
\qquad
2\le \ell'_1<\cdots<\ell'_b\le K/2
\]
be even integers. If $a+b\le d$, then
\[
\rho_{\ell_1},\ldots,\rho_{\ell_a},
\rho_{\ell'_1-1},\ldots,\rho_{\ell'_b-1}
\]
are linearly independent. Thus Conjecture~1.2 of Xue
\cite{XueMixed} holds.
\end{corollary}
\begin{proof}
The displayed indices are distinct because the first group is even and the
second is odd. They all lie strictly below $h=(K-2)/2$, except possibly
$\ell'_b-1=h$. The latter can occur only when $\ell'_b=K/2$ is even, hence
$K\equiv0\pmod4$; Corollary~\ref{cor:full-spark-periods} allows exactly this
endpoint. The result follows.
\end{proof}

\begin{corollary}[Even-period bases]\label{cor:periods}
If $4\le b_1<\cdots<b_d<(K-2)/2$ are even, then
$\rho_{b_1},\ldots,\rho_{b_d}$ form a basis of $S_K^*$.
In particular, $\rho_4,\rho_6,\ldots,\rho_{2d+2}$ form such a basis.
\end{corollary}
\begin{proof}
The first assertion is Corollary~\ref{cor:full-spark-periods} with $m=d$.
For the consecutive choice,
$K-2-2(2d+2)=8d+\kappa-6>0$, so $2d+2<(K-2)/2$.
\end{proof}

\begin{remark}[Period polynomials]
For $P_f(X)=\int_0^{i\infty}f(z)(X-z)^{K-2}\,dz$, its even part is
\[
P_f^+(X)=\sum_{\substack{0\le t\le K-2\\t\ \mathrm{even}}}
\binom{K-2}{t}\rho_t(f)X^{K-2-t}.
\]
Corollary~\ref{cor:periods} says that the coefficients corresponding to
$t=4,6,\ldots,2d+2$ determine $f$. Corollary~\ref{cor:full-spark-periods}
is stronger: it permits arbitrary mixtures of odd and even period indices in
the fundamental half-range, subject only to the dimension bound. These are
statements about linear functionals on $S_K$, not algebraic independence of
numerical period values of an individual form.
\end{remark}

\section{The published period pairing and its sign}
We use the period--bracket correspondence of Lei--Ni--Xue
\cite[Propositions 1.3 and 4.1]{LNX}. Let $\mathcal B_K$ be a basis
of normalized Hecke eigenforms in $S_K$. Their Fourier coefficients
are real, and the forms are mutually orthogonal. Normalized here means
that the coefficient of $q$ is $1$; the product of two periods below
is asserted only for such eigenforms.

\begin{lemma}[Signed period pairing]\label{lem:Rankin}
Let $a>b\ge4$ be even, $a+b+2=K$, and $g\in\mathcal B_K$. Then
\begin{equation}\label{eq:RankinI}
\langle g,[E_a,E_b]_1\rangle
=c_{a,b}L(g,K-2)I_a(g),
\end{equation}
where $c_{a,b}$ is a nonzero real number independent of $g$, and
\begin{equation}\label{eq:c}
\sgn c_{a,b}=(-1)^{b/2}.
\end{equation}
\end{lemma}
\begin{proof}
The Eisenstein series and normalized derivatives in \cite{LNX} agree with
\eqref{eq:E} and \eqref{eq:RCintro}. Their period convention is also
exactly ours: if $r_t(g)$ denotes the period in \cite{LNX}, then
\[
r_t(g)=\int_0^{i\infty}g(z)z^t\,dz
=\frac{t!}{(-2\pi i)^{t+1}}L(g,t+1)
=i^{t+1}\frac{\Gamma(t+1)}{(2\pi)^{t+1}}L(g,t+1)
=\rho_t(g).
\]
Thus no additional power of $2\pi$ or $i$ is hidden in the passage from
\cite[Proposition~1.3]{LNX} to our notation. With bracket order $1$, that
proposition gives
\begin{equation}\label{eq:LNXperiod}
\begin{split}
\langle g,[E_a,E_b]_1\rangle
={}&(-i)^{a+1}\frac{4ab}{B_aB_b}
\frac{(-1)^{1+b/2}\Gamma(K-1)}{2^{K-1}\Gamma(K-2)}\\
&\hspace{12mm}\times\rho_{K-3}(g)\rho_b(g).
\end{split}
\end{equation}
Use $\rho_b(g)=-\rho_a(g)$ and
\[
\rho_{K-3}(g)=i^{K-2}
\frac{\Gamma(K-2)}{(2\pi)^{K-2}}L(g,K-2),
\qquad \rho_a(g)=i^{a+1}I_a(g).
\]
The resulting phase is
\[
-(-1)^{1+b/2}(-i)^{a+1}i^{K-2}i^{a+1}
=(-1)^{a/2}.
\]
Consequently the real factor in \eqref{eq:RankinI} is
\begin{equation}\label{eq:cexact}
c_{a,b}=(-1)^{a/2}
\frac{4ab\,\Gamma(K-1)}
{B_aB_b\,2^{K-1}(2\pi)^{K-2}}.
\end{equation}
It is nonzero. Since
$\sgn(B_aB_b)=(-1)^{a/2+b/2}$, its sign is $(-1)^{b/2}$.
This is the only information about its magnitude needed in Section~6.
\end{proof}

\begin{lemma}\label{lem:positiveL}
For every $g\in\mathcal B_K$, $L(g,K-2)>0$.
\end{lemma}
\begin{proof}
It suffices to use the elementary bound $a_g(n)=O_g(n^{K/2})$ for
cusp-form coefficients. Indeed $y^{K/2}|g(z)|$ is bounded, and the
Fourier coefficient integral at $y=1/n$ gives this bound. In particular,
for a fixed prime $p$ the local series
\[
\sum_{m\ge0}a_g(p^m)X^m
=\frac{1}{1-a_g(p)X+p^{K-1}X^2}
\]
converges for $|X|<p^{-K/2}$. Hence its denominator has no zero in that
open disk. If $s>K/2$, then $X=p^{-s}$ lies in the interval
$(0,p^{-K/2})$; the real-valued denominator is therefore nonzero there
and, being equal to $1$ at $X=0$, is strictly positive throughout that
interval. Thus every local Euler factor at such a real $s$ is positive.

For $s>K/2+1$ the Euler product is absolutely convergent. Indeed
$a_g(p)=O_g(p^{K/2})$, so
$\sum_p|a_g(p)|p^{-s}<\infty$, while
$\sum_pp^{K-1-2s}<\infty$. Equivalently, the deviations of the local
factors from $1$ are absolutely summable; the product therefore converges
to a finite nonzero number. Since each factor is positive, the product is
strictly positive. Finally $K-2>K/2+1$ for $K\ge12$, so
$L(g,K-2)>0$. In particular no bound as strong as Deligne's is needed.
\end{proof}

Define a real linear operator on $S_K(\RR)$ by
\begin{equation}\label{eq:T}
\Tcal g=L(g,K-2)g\qquad(g\in\mathcal B_K).
\end{equation}
Lemma~\ref{lem:positiveL} implies
\begin{equation}\label{eq:detT}
\det\Tcal=\prod_{g\in\mathcal B_K}L(g,K-2)>0.
\end{equation}
Linearity of \eqref{eq:RankinI} in $g$ gives, for every
$f\in S_K(\RR)$,
\begin{equation}\label{eq:allfRankin}
\langle f,[E_a,E_b]_1\rangle=c_{a,b}I_a(\Tcal f).
\end{equation}
This step keeps the eigenform-dependent factor $L(g,K-2)$ inside a
single common invertible operator; it is not pulled out as a scalar
independent of the eigenform.

\section{Signed basis determinants in arbitrary weight}
\begin{theorem}[Arbitrary admissible families]\label{thm:general}
Let $K\ge12$ be even and $d=\dim S_K>0$. Choose even integers
\begin{equation}\label{eq:admissible}
4\le b_1<\cdots<b_d<\frac{K-2}{2},\qquad a_j=K-2-b_j.
\end{equation}
Then the forms $[E_{a_j},E_{b_j}]_1$ form a basis of $S_K$, and
\begin{equation}\label{eq:generalSign}
(-1)^{\sum_{j=1}^db_j/2}
\det\bigl([q^n][E_{a_j},E_{b_j}]_1\bigr)_{n,j=1}^d>0.
\end{equation}
\end{theorem}
\begin{proof}
The basis assertion follows already from Corollary~\ref{cor:periods}
and \cite[Proposition 4.1]{LNX}. To determine the orientation, we use
Lemma~\ref{lem:Rankin} and retain the signed column factors.
Let $X$ be the real coordinate matrix in the basis \eqref{eq:Fbasis},
so that $[E_{a_j},E_{b_j}]_1=\sum_mX_{mj}\Phi_m$. Let
\[
G=(\langle\Phi_m,\Phi_n\rangle)_{m,n},\quad
A=(I_{a_j}(\Phi_m))_{m,j},\quad C=\diag(c_{a_j,b_j}),
\]
and let $T$ be the matrix of $\Tcal$ with columns representing the
images of the basis vectors. Equation~\eqref{eq:allfRankin} gives
\begin{equation}\label{eq:matrixid}
GX=T^{\mathsf T}AC.
\end{equation}
All matrices are real; this avoids any conjugate-linearity ambiguity
in using the Petersson product. The Gram determinant satisfies
$\det G>0$, and \eqref{eq:detT} gives $\det T>0$.
Moreover $a_1>\cdots>a_d>(K-2)/2$, so
Theorem~\ref{thm:perioddet} gives $\det A>0$. Therefore
\begin{equation}\label{eq:detX}
\det X=\frac{\det T\det A}{\det G}\prod_{j=1}^dc_{a_j,b_j}\ne0.
\end{equation}
Its sign is $(-1)^{\sum_jb_j/2}$ by \eqref{eq:c}.

The first-$d$ Fourier-coefficient matrix of $\Phi_1,\ldots,\Phi_d$ is
unit lower triangular, since $\Phi_m=q^m+O(q^{m+1})$.
Calling it $Q$, the matrix in \eqref{eq:generalSign} is $QX$ and
$\det Q=1$. Thus its determinant has the same nonzero sign as
$\det X$, proving both claims.
\end{proof}

\begin{proof}[Proof of Theorem~\ref{thm:mainbasis}]
Use $b_j=2j+2$, $a_j=K-2-b_j$. By \eqref{eq:kappa},
\[
a_d-b_d=K-4d-6=8d+\kappa-6>0.
\]
Thus all indices are admissible. Finally
\[
\sum_{j=1}^d\frac{b_j}{2}=\frac{d(d+3)}2
=\binom d2+2d,
\]
so Theorem~\ref{thm:general} gives \eqref{eq:mainsign}.
\end{proof}

\begin{corollary}\label{cor:subfamilies}
Any admissible family of $s\le d$ first brackets, with distinct even
$b$-indices in $[4,(K-2)/2)$, is linearly independent. The corresponding
even periods are linearly independent as well.
\end{corollary}
\begin{proof}
The consecutive choice in Theorem~\ref{thm:mainbasis}'s proof shows
that the admissible index set has at least $d$ elements for every $K$
under consideration. Enlarge the chosen set to $d$ admissible indices
and apply Theorem~\ref{thm:general} and Corollary~\ref{cor:periods}.
\end{proof}

\begin{remark}
Corollary~\ref{cor:subfamilies} proves the arbitrary-family assertion
stated as \cite[Conjecture 4.2]{NSSXY}. In contrast to a spanning
argument, \eqref{eq:detX} supplies a nonzero determinant with a specified
orientation. The proof is special to the level-one structure used in
\eqref{eq:Fbasis}; no higher-level extension is asserted here.
\end{remark}

\section{Exact Fourier coordinates and the normalized system}
We now return to $K=2r+14$ and the consecutive indices \eqref{eq:ab}.
For $f=\sum_{u\ge0}A(u)q^u$ and $g=\sum_{v\ge0}B(v)q^v$ of weights
$a,b$, coefficient multiplication gives
\[
[q^n][f,g]_1=\sum_{u=0}^n\bigl(an-(a+b)u\bigr)A(u)B(n-u).
\]
For Eisenstein series it becomes
\begin{align}
[q^n][E_a,E_b]_1={}&ae_b\,n\sigma_{b-1}(n)-be_a\,n\sigma_{a-1}(n)
\notag\\
&+e_ae_b\sum_{m=1}^{n-1}(an-(a+b)m)
\sigma_{a-1}(m)\sigma_{b-1}(n-m).
\label{eq:RCcoeff}
\end{align}
Consequently
\begin{equation}\label{eq:hnorm}
h_{r,j}(n)=\frac{B_{a_j}B_{b_j}}{4a_jb_j}[q^n]R_j.
\end{equation}
This proves directly that \eqref{eq:h} is a rational, explicit divisor-sum
formula for every entry of $N_r$.

\begin{proposition}[Unique rational coordinates]\label{prop:coordinates}
There exist unique rational coefficients $\lambda_j(r)$ such that
\begin{equation}\label{eq:Frbasis}
E_{2r+2}\Delta=\sum_{j=1}^d\lambda_j(r)R_j.
\end{equation}
If $v_r=([q^n](E_{2r+2}\Delta))_{n=1}^d$, then
\begin{equation}\label{eq:lambda}
\lambda=M_r^{-1}v_r,\qquad
\lambda_j=\frac{\det M_r[j\leftarrow v_r]}{\det M_r}.
\end{equation}
Writing
\begin{equation}\label{eq:Zlambda}
Z_j(r)=\frac{4a_jb_jB_{2r+2}}{B_{a_j}B_{b_j}}\lambda_j(r),
\end{equation}
one has $N_rZ=z_r$, hence \eqref{eq:Zcramer}. Also
\begin{equation}\label{eq:Nsign}
\sgn\det N_r=(-1)^{\binom d2+d(K/2-1)}.
\end{equation}
\end{proposition}
\begin{proof}
Theorem~\ref{thm:mainbasis} gives \eqref{eq:Frbasis}. Comparing the
first $d$ coefficients gives \eqref{eq:lambda}, and all entries are
rational. Put
\[
S=\diag\left(\frac{B_{a_j}B_{b_j}}{4a_jb_j}\right)_{j=1}^d.
\]
Equation~\eqref{eq:hnorm} gives $N_r=M_rS$, whereas
\eqref{eq:Zlambda} gives $SZ=B_{2r+2}\lambda$. Thus
$N_rZ=B_{2r+2}v_r=z_r$. Each diagonal entry of $S$ has sign
\[
(-1)^{a_j/2+1+b_j/2+1}=(-1)^{K/2-1}.
\]
Combining with \eqref{eq:mainsign} proves \eqref{eq:Nsign}.
\end{proof}

\begin{remark}[Why the coordinate construction is not circular]\label{rem:noncircular}
Only the finitely many coefficients of $E_{2r+2}\Delta$ through $q^d$
enter $z_r$, regardless of the target $n$. They can be found by direct
power-series multiplication. More explicitly,
\begin{equation}\label{eq:ztarget}
z_{r,n}=B_{2r+2}\tau(n)
-4(r+1)\sum_{m=1}^{n-1}\sigma_{2r+1}(m)\tau(n-m),\quad 1\le n\le d.
\end{equation}
The values $\tau(1),\ldots,\tau(d)$ can themselves be generated from
divisor sums alone by
\begin{equation}\label{eq:taurec}
\tau(1)=1,\qquad
(n-1)\tau(n)=-24\sum_{m=1}^{n-1}\sigma_1(m)\tau(n-m)\quad(n\ge2).
\end{equation}
Indeed $D\Delta=E_2\Delta$ gives this recurrence upon comparing
coefficients. Thus all $Z_j(r)$ are precomputed in a finite procedure
depending on $r$ only. No knowledge of $C_r(n)$ at a general target
$n>d$ is needed to determine them.
\end{remark}

\section{The all-weight convolution identity}
\begin{proof}[Proof of Theorem~\ref{thm:mainconv}]
By \eqref{eq:Frbasis}, \eqref{eq:hnorm}, and \eqref{eq:Zlambda},
\begin{equation}\label{eq:coeffsum}
\sum_{j=1}^dZ_j(r)h_{r,j}(n)
=B_{2r+2}[q^n](E_{2r+2}\Delta)
\end{equation}
for every $n\ge1$. Proposition~\ref{prop:extract} gives
\eqref{eq:compactC}. The Schur-complement identity gives
\[
\det\begin{pmatrix}N_r&z_r\\h_r(n)^{\mathsf T}&B_{2r+2}\tau(n)\end{pmatrix}
=\det N_r\bigl(B_{2r+2}\tau(n)-h_r(n)^{\mathsf T}N_r^{-1}z_r\bigr),
\]
which proves \eqref{eq:border}. The nonsingularity of $N_r$ was proved
in Proposition~\ref{prop:coordinates}, not assumed.
\end{proof}

\begin{corollary}[Expanded divisor-sum form]\label{cor:expanded}
With $Z_j(r)$ specified by \eqref{eq:Zcramer}, for all $r,n\ge1$,
\begin{align}
C_r(n)={}&\frac{B_{2r+2}}{4(r+1)}\tau(n)
+\sum_{j=1}^d\frac{Z_j(r)B_{a_j}}{8(r+1)}\,n\sigma_{b_j-1}(n)
\notag\\
&-\sum_{j=1}^d\frac{Z_j(r)B_{b_j}}{8(r+1)}\,n\sigma_{a_j-1}(n)
\notag\\
&-\frac1{2(r+1)}\sum_{j=1}^dZ_j(r)
\sum_{m=1}^{n-1}\bigl((r+5-j)n-(r+6)m\bigr)
\notag\\[-2pt]
&\hspace{38mm}\times\sigma_{a_j-1}(m)\sigma_{b_j-1}(n-m).
\label{eq:expanded}
\end{align}
\end{corollary}
\begin{proof}
Insert \eqref{eq:h} into \eqref{eq:compactC}. The first two terms of
$h_{r,j}(n)$ give the two single-divisor terms with the displayed signs.
For the double sum, use
\[
a_jn-(a_j+b_j)m
=2\bigl((r+5-j)n-(r+6)m\bigr).
\]
This accounts for the factor $1/(2(r+1))$ and proves the formula.
\end{proof}

\begin{remark}[The precise improvement over the finite theorem]
The algebraic shape of \eqref{eq:expanded} is the same as in the
finite-dimensional calculation. What is added is a proof that its
coefficients exist uniquely for every $r\ge1$, together with a nonzero
signed denominator. The formula holds for every $n$ because
\eqref{eq:Frbasis} is an identity of modular forms; fitting finitely many
Fourier coefficients is not being substituted for this identity.
\end{remark}

\begin{remark}[Scope of the explicit formula]
For a fixed $r$, the formula replaces the mixed divisor--tau convolution
by the single value $\tau(n)$, ordinary divisor sums, and weighted
convolutions of divisor sums. It does not eliminate $\tau(n)$ itself.
Nor does the sign theorem imply a sign or a short scalar formula for
each $Z_j(r)$: its Cramer numerator is a different determinant. The
determinantal specification is exact in all dimensions, although its
size increases with $r$.
\end{remark}

\section{Worked examples and the low-dimensional specializations}
The examples below are deductions from the general theorem. They also
fix concrete signs and normalizations that can be checked directly.

\subsection{Dimension one}
Here $r\in\{1,2,3,4,6\}$ and
$R_1=[E_{2r+8},E_4]_1$. Its first coefficient is
\[
[q]R_1=\frac{16(r+4)}{B_{2r+8}}(30B_{2r+8}+1).
\]
Since $[q](E_{2r+2}\Delta)=1$, we obtain
\begin{equation}\label{eq:d1lambda}
\lambda_1(r)=\frac{B_{2r+8}}{16(r+4)(30B_{2r+8}+1)},\qquad
Z_1(r)=-\frac{60B_{2r+2}}{30B_{2r+8}+1}.
\end{equation}
The denominator is nonzero by Theorem~\ref{thm:mainbasis}.
Substitution in \eqref{eq:expanded} gives
\begin{align*}
C_r(n)={}&\frac{B_{2r+2}}{4(r+1)}\tau(n)
-\frac{15B_{2r+2}B_{2r+8}}{2(r+1)(30B_{2r+8}+1)}\,n\sigma_3(n)\\
&-\frac{B_{2r+2}}{4(r+1)(30B_{2r+8}+1)}\,n\sigma_{2r+7}(n)\\
&+\frac{30B_{2r+2}}{(r+1)(30B_{2r+8}+1)}
\sum_{m=1}^{n-1}\bigl((r+4)n-(r+6)m\bigr)
\sigma_{2r+7}(m)\sigma_3(n-m).
\end{align*}
In particular $r=1$ gives $\lambda_1=1/3456$, $Z_1=11/18$, and
\begin{align}
C_1(n)={}&-\frac1{240}\tau(n)+\frac5{1728}\,n\sigma_3(n)
+\frac{11}{8640}\,n\sigma_9(n)\notag\\
&-\frac{11}{72}\sum_{m=1}^{n-1}(5n-7m)\sigma_9(m)\sigma_3(n-m).
\label{eq:C1}
\end{align}
For $n=1$ its right side is $-1/240+5/1728+11/8640=0$, as required.

\subsection{Dimension two}
Here $r\in\{5,7,8,9,10,12\}$. To retain a compact scalar version,
put $U=B_{2r+8}$, $V=B_{2r+6}$, and $t=4^r$, and write
\begin{align*}
u_1&=\frac{16(r+4)}U(30U+1),&
u_2&=\frac{24(r+3)}V(1-42V),\\
v_1&=\frac{32(r+4)}U(270U+128\cdot t-60r-119),\\
v_2&=\frac{48(r+3)}V(-1386V+32\cdot t+84r+1).
\end{align*}
These are the first two coefficients of $R_1$ and $R_2$, respectively.
Set $w=-24-4(r+1)/B_{2r+2}$. Then
\begin{equation}\label{eq:d2}
\lambda_1=\frac{v_2-u_2w}{u_1v_2-u_2v_1},\qquad
\lambda_2=\frac{u_1w-v_1}{u_1v_2-u_2v_1}.
\end{equation}
The denominator is strictly negative by \eqref{eq:mainsign}.
These formulas follow directly from the $2\times2$ coordinate system
and recover the expanded scalar formulas directly without introducing additional exceptional denominators.

For $r=5$, exact simplification gives
\begin{equation}\label{eq:r5coords}
(Z_1,Z_2)=\left(\frac{95}{5304},\frac7{104}\right),\qquad
(\lambda_1,\lambda_2)=
\left(\frac{219335}{487171584},\frac{25319}{216520704}\right).
\end{equation}
Thus the convolution $\sigma_{11}*\tau$ is
\begin{align}
C_5(n)={}&-\frac{691}{65520}\tau(n)
+\frac{219335}{10692864}\,n\sigma_3(n)
-\frac{25319}{2545920}\,n\sigma_5(n)\notag\\
&+\frac{19}{1527552}\,n\sigma_{17}(n)
-\frac1{29952}\,n\sigma_{15}(n)\notag\\
&-\frac{95}{63648}\sum_{m=1}^{n-1}(9n-11m)
\sigma_{17}(m)\sigma_3(n-m)\notag\\
&-\frac7{1248}\sum_{m=1}^{n-1}(8n-11m)
\sigma_{15}(m)\sigma_5(n-m).
\label{eq:C5}
\end{align}
This is exactly the corresponding specialization of the general formula.

\subsection{A dimension-three example}
For $r=11$ we have $K=36$ and $d=3$. The general rational system gives
\begin{equation}\label{eq:r11coords}
(Z_1,Z_2,Z_3)=
\left(\frac{3386471}{4177258020},\frac{27566}{2118285},
\frac{316393}{11080260}\right).
\end{equation}
Here $(a_1,a_2,a_3)=(30,28,26)$ and $(b_1,b_2,b_3)=(4,6,8)$.
Together with \eqref{eq:expanded}, these three rational numbers give
the identity for $C_{11}(n)$ for all $n$, without printing a large
expanded polynomial numerator. The determinant sign is negative.

\subsection{A higher-dimensional example}
The first parameter with $d=9$ is $r=47$, for which $K=108$.
The basis now contains the nine brackets
\[
[E_{102},E_4]_1,[E_{100},E_6]_1,\ldots,[E_{86},E_{20}]_1.
\]
Theorem~\ref{thm:mainbasis} gives a strictly positive coefficient
determinant, since $\binom92=36$. Equations \eqref{eq:h},
\eqref{eq:Nz}, and \eqref{eq:Zcramer} specify the nine coefficients
of the identity for $C_{47}(n)$ exactly. No new independence calculation
is required at this dimension or at any subsequent dimension.

\section{Unique canonical all-weight congruences}\label{sec:uniquecong}

The characteristic-zero formula of Theorem~\ref{thm:mainconv} already gives a
preferred coordinate system. The purpose of this section is to show that, away
from a finite set of primes, this coordinate system survives reduction without any
loss of rank. The resulting modular congruence is therefore not obtained by fitting
a finite list of values of $n$: it is the reduction of an identity of modular forms,
and its coefficient vector is forced by the reduced basis.

Throughout this section put
\[
K=2r+14,\qquad s=2r+2=K-12,\qquad d=\dim S_K=\dim M_s,
\]
and retain the consecutive pairs
\[
a_j=2r+10-2j,\qquad b_j=2j+2,\qquad 1\le j\le d.
\]
Set
\begin{equation}\label{eq:all-blocks}
H_j=[G_{a_j},G_{b_j}]_1,\qquad
T_r=-2sG_s\Delta=B_sE_s\Delta.
\end{equation}

For a prime $p$, write $R_p=\mathbf Z_{(p)}$ and define the cusp-form lattice
\begin{equation}\label{eq:integral-lattice}
\mathcal S_{K,p}:=\bigoplus_{m=1}^d R_p\Phi_m\subset S_K(\mathbf Q),
\qquad
\overline{\mathcal S}_{K,p}:=\mathcal S_{K,p}/p\mathcal S_{K,p}.
\end{equation}
Thus $\overline{\mathcal S}_{K,p}$ is a $d$-dimensional $\mathbf F_p$-space.
Because the first-$d$ coefficient matrix of the $\Phi_m$ is unit lower
triangular over $\mathbf Z$, $\mathcal S_{K,p}$ is equivalently the set of
rational cusp forms of weight $K$ whose Fourier coefficients are $p$-integral.
All reductions in this section take place in
$\overline{\mathcal S}_{K,p}$; no identification with the full Katz space of
mod-$p$ modular forms is required.

\begin{lemma}[The all-weight identity as a modular-form coordinate identity]
\label{lem:all-coordinate}
For every $j$ and every $n\ge1$,
\[
[q^n]H_j=h_{r,j}(n).
\]
Moreover
\begin{equation}\label{eq:all-coordinate}
T_r=\sum_{j=1}^d Z_j(r)H_j.
\end{equation}
\end{lemma}

\begin{proof}
Since $G_k=-B_kE_k/(2k)$ and the first Rankin--Cohen bracket is bilinear,
\[
H_j=\frac{B_{a_j}B_{b_j}}{4a_jb_j}[E_{a_j},E_{b_j}]_1.
\]
Equation~\eqref{eq:hnorm} therefore gives the first assertion. Proposition~2.2
gives
\[
[q^n](B_sE_s\Delta)=B_s\tau(n)-2sC_r(n),
\]
while $2s=4(r+1)$. Substituting Theorem~\ref{thm:mainconv} proves equality
of every positive Fourier coefficient. Both sides are cusp forms of weight
$K$, hence \eqref{eq:all-coordinate} follows.
\end{proof}

\begin{lemma}[The first $d$ coefficients are integral coordinates]
\label{lem:first-d-coordinates}
The map
\[
\mathcal S_{K,p}\longrightarrow R_p^d,
\qquad
F\longmapsto ([q]F,\ldots,[q^d]F),
\]
is an $R_p$-module isomorphism. In particular, $d$ forms in
$\mathcal S_{K,p}$ form an $R_p$-basis if and only if their first-$d$ Fourier
coefficient matrix has unit determinant.
\end{lemma}

\begin{proof}
Each $\Phi_m$ satisfies $\Phi_m=q^m+O(q^{m+1})$. Hence the matrix of the first
$d$ positive Fourier coefficients of $\Phi_1,\ldots,\Phi_d$ is unit lower
triangular over $\mathbf Z$, and therefore over $R_p$. This proves the asserted
coordinate isomorphism and determinant criterion. It also justifies the
$q$-expansion characterization of $\mathcal S_{K,p}$ stated above.
\end{proof}

\begin{definition}[All-weight good primes]
A prime $p$ is \emph{coefficient-admissible} for $r$ if the rational Fourier
coefficients of $T_r,H_1,\ldots,H_d$ are $p$-integral and $2s$ is a
$p$-adic unit. Equivalently, the target and canonical blocks lie in
$\mathcal S_{K,p}$ and $2s\in R_p^\times$. It is \emph{all-weight good} if in addition
\begin{equation}\label{eq:good-prime}
\det N_r\not\equiv0\pmod p.
\end{equation}
\end{definition}

\begin{proposition}[A uniform safe range]\label{prop:safe-range}
Every prime
\begin{equation}\label{eq:safe-range}
p\ge K-3
\end{equation}
is coefficient-admissible. Consequently a prime in this range is all-weight
good exactly when it does not divide the numerator of $\det N_r$ in lowest terms.
\end{proposition}

\begin{proof}
The Bernoulli numbers involved in $T_r$ and the $H_j$ have even indices among
$s,a_1,\ldots,a_d,b_1,\ldots,b_d$. Their largest index is
$a_1=K-6$. If $p\ge K-3$, then $p-1>K-6$, so the von
Staudt--Clausen theorem shows that $p$ occurs in none of the relevant Bernoulli
denominators. Moreover $a_j,b_j\le K-6<p$ and $s=K-12<p$. Hence $p$ divides none
of the integer factors $2a_j$, $2b_j$, or $2s$.

It is useful to make explicit why the normalization $G_k=-B_kE_k/(2k)$
introduces no hidden denominator. From \eqref{eq:E},
\begin{equation}\label{eq:G-integral-expansion}
G_k=-\frac{B_k}{2k}+\sum_{n\ge1}\sigma_{k-1}(n)q^n.
\end{equation}
Thus all positive Fourier coefficients of $G_k$ are integers, and in the
present range its constant term is $p$-integral because neither the denominator
of $B_k$ nor $2k$ is divisible by $p$. It follows directly that every
$H_j=[G_{a_j},G_{b_j}]_1$ is $p$-integral, since $D$ and the integral weight
factors in the first bracket preserve $p$-integrality. Likewise
\[
T_r=B_sE_s\Delta
=\left(B_s-2s\sum_{n\ge1}\sigma_{s-1}(n)q^n\right)\Delta
\]
is $p$-integral. Finally $2s$ is a $p$-adic unit. Therefore the entries of
$N_r$ and $z_r$ are $p$-integral, and in this range \eqref{eq:good-prime}
is the only remaining condition.
\end{proof}

\begin{proposition}[Equivalent forms of the good-prime condition]
\label{prop:good-equivalences}
For a coefficient-admissible prime $p$, the following are equivalent:
\begin{enumerate}
\item $p$ is all-weight good;
\item $\det N_r$ is a $p$-adic unit;
\item the reductions $\overline{H_1},\ldots,\overline{H_d}$ form a basis of
$\overline{\mathcal S}_{K,p}$;
\item the target $\overline{T_r}$ has a unique coordinate vector in the
reduced canonical blocks.
\end{enumerate}
\end{proposition}

\begin{proof}
By Lemma~\ref{lem:first-d-coordinates}, the matrix $N_r$ is exactly the
first-$d$ coefficient matrix of the canonical blocks $H_j$. Thus (1)--(3) are
equivalent. A basis gives unique coordinates for every vector, proving (3)$\Rightarrow$(4).
Conversely, if the $H_j$ were linearly dependent modulo $p$, then any coordinate
vector of $T_r$ could be altered by a nonzero relation among the $H_j$, contradicting
uniqueness. Hence (4)$\Rightarrow$(3).
\end{proof}

\begin{theorem}[Unique canonical congruence]\label{thm:unique-congruence}
Let $p$ be all-weight good. Then there exists a unique vector
$(D_1,\ldots,D_d)\in\mathbf F_p^d$ such that, for every $n\ge1$,
\begin{equation}\label{eq:unique-congruence}
C_r(n)\equiv
\frac{B_s}{2s}\tau(n)
-\frac1{2s}\sum_{j=1}^dD_jh_{r,j}(n)\pmod p.
\end{equation}
Moreover
\begin{equation}\label{eq:unique-coordinates}
D_j\equiv Z_j(r)\pmod p\qquad(1\le j\le d).
\end{equation}
\end{theorem}

\begin{proof}
By Proposition~\ref{prop:good-equivalences}, the reductions
$\overline{H_1},\ldots,\overline{H_d}$ form a basis of
$\overline{\mathcal S}_{K,p}$. The characteristic-zero identity
\eqref{eq:all-coordinate} is defined over $\mathbf Z_{(p)}$, so it may be reduced
coefficientwise modulo $p$. This gives
\[
\overline{T_r}=\sum_{j=1}^d\overline{Z_j(r)}\,\overline{H_j}.
\]
Because the reduced blocks form a basis, this coordinate vector is the only one
possible. Taking the coefficient of $q^n$ and using
$[q^n]T_r=B_s\tau(n)-2sC_r(n)$ gives \eqref{eq:unique-congruence} for every
$n\ge1$ simultaneously. No Sturm-bound argument or interpolation in the variable
$n$ is needed: the congruence is the Fourier expansion of a single modular-form
identity.
\end{proof}

\begin{remark}[Meaning of uniqueness]
For fixed $r$ and a fixed good prime $p$, Theorem~\ref{thm:unique-congruence}
asserts uniqueness of the coefficient vector in the canonical consecutive blocks.
Different good primes generally give different reductions of the same rational
Cramer vector, and a noncanonical choice of modular forms may lead to a different
looking identity. The theorem is therefore a basis-theoretic uniqueness statement,
not a claim that the displayed congruence is the only conceivable formula modulo
$p$.
\end{remark}

\begin{remark}[Scope of the mod-$p$ statement]
The space $\overline{\mathcal S}_{K,p}$ is the reduction of a fixed
characteristic-zero $q$-expansion lattice. The theorem therefore asserts
uniqueness inside this canonical reduction, which is exactly what is needed to
reduce \eqref{eq:all-coordinate}. No claim about equality with a larger
characteristic-$p$ moduli-theoretic space is used or required.
\end{remark}

\begin{corollary}[Infinitely many unique congruences in every weight]
\label{cor:infinitely-many}
For every $r\ge1$, all but finitely many primes $p\ge K-3$ are all-weight good.
Consequently every $C_r$ admits infinitely many prime-wise unique canonical
all-weight congruences of the form \eqref{eq:unique-congruence}.
\end{corollary}

\begin{proof}
By Theorem~\ref{thm:mainconv}, $D_r:=\det N_r$ is a nonzero rational number.
Write $D_r=u/v$ in lowest terms. Only the finitely many primes dividing $u$ can
violate the determinant condition once coefficient-admissibility is known. By
Proposition~\ref{prop:safe-range}, every prime $p\ge K-3$ is already
coefficient-admissible, so every such prime with $p\nmid u$ is good. There are
infinitely many primes and only finitely many excluded divisors of $u$; hence
infinitely many good primes remain. Applying Theorem~\ref{thm:unique-congruence}
to each of them proves the claim.
\end{proof}

\section{Cramer coordinates and simultaneous coefficient vanishing}
\label{sec:cramer}

The preceding section proves existence and uniqueness for every good prime.
The remaining arithmetic of a particular congruence is therefore entirely
coordinatewise: a canonical block disappears exactly when its unique Cramer
coordinate becomes zero modulo $p$. We now record this fact in forms that are
useful both for individual coordinates and for simultaneous vanishing.

Put
\begin{equation}\label{eq:cramerAD}
D_r=\det N_r,\qquad
A_{r,j}=\det N_r[j\leftarrow z_r].
\end{equation}
Then $Z_j=A_{r,j}/D_r$. Write each coordinate in lowest terms as
\begin{equation}\label{eq:UV}
Z_j(r)=\frac{U_j(r)}{V_j(r)},
\qquad
U_j(r)\in\mathbf Z,\quad V_j(r)\in\mathbf Z_{>0},
\quad \gcd(U_j,V_j)=1.
\end{equation}

For a fixed good prime $p$, Cramer's rule may be read locally in
$\mathbf Z_{(p)}$. Since $D_r$ is a unit there,
\[
Z_j=\frac{A_{r,j}}{D_r}\in\mathbf Z_{(p)},
\]
and hence the residue of $Z_j$ vanishes if and only if $A_{r,j}$ lies in
$p\mathbf Z_{(p)}$. Passing to the globally reduced fraction $U_j/V_j$ removes
all common numerator--denominator factors and yields the following intrinsic
criterion.

\begin{theorem}[Cramer numerator criterion]\label{thm:cramer-numerator}
If $p$ is all-weight good and $(D_1,\ldots,D_d)$ is the unique coefficient vector
from Theorem~\ref{thm:unique-congruence}, then every $V_j(r)$ is a $p$-adic unit
and
\begin{equation}\label{eq:cramer-zero}
\boxed{\;D_j=0\quad\Longleftrightarrow\quad p\mid U_j(r).\;}
\end{equation}
Thus the reduced Cramer numerators record exactly which canonical terms disappear
from the unique congruence modulo $p$.
\end{theorem}

\begin{proof}
For a good prime, $D_r$ is a $p$-adic unit and every entry of $N_r$ and $z_r$
is $p$-integral. Thus $A_{r,j}\in\mathbf Z_{(p)}$ and
$Z_j=A_{r,j}/D_r\in\mathbf Z_{(p)}$. If the lowest-term denominator $V_j$
were divisible by $p$, then $Z_j$ would fail to be $p$-integral, a contradiction;
so $V_j$ is a unit. Therefore
\[
Z_j\equiv0\pmod p
\quad\Longleftrightarrow\quad
U_j\equiv0\pmod p.
\]
Combining this with \eqref{eq:unique-coordinates} proves
\eqref{eq:cramer-zero}.
\end{proof}

\begin{remark}[Characteristic-zero zero coordinates]
If some $U_j(r)$ is already zero over $\mathbf Q$, then the $j$th canonical
block is absent from the characteristic-zero all-weight identity itself and hence
from every good-prime reduction. In that case all statements below apply after
separating these fixed zero coordinates from the nonzero $U_j$. The generic
formulation \eqref{eq:Gamma-m} is written under the simplifying assumption that
no $U_j$ vanishes identically.
\end{remark}

Assume first that none of the $U_j$ is zero. For $1\le m\le d$ define
\begin{equation}\label{eq:Gamma-m}
\Gamma_m(r)
=
\gcd_{\substack{J\subseteq\{1,\ldots,d\}\\|J|=d-m+1}}
\prod_{j\in J}|U_j(r)|.
\end{equation}
Thus
\[
\Gamma_1(r)=\prod_{j=1}^d|U_j(r)|,\qquad
\Gamma_d(r)=\gcd(U_1(r),\ldots,U_d(r)).
\]

\begin{theorem}[Exact simultaneous-vanishing layers]\label{thm:Gamma-layer}
Let $p$ be all-weight good and let $(D_1,\ldots,D_d)$ be the unique congruence
coefficients. For $1\le m\le d$,
\begin{equation}\label{eq:Gamma-layer}
\boxed{\;
p\mid\Gamma_m(r)
\quad\Longleftrightarrow\quad
\#\{j:D_j=0\}\ge m.
\;}
\end{equation}
Consequently
\[
\operatorname{Supp}\Gamma_d
\subseteq\cdots\subseteq
\operatorname{Supp}\Gamma_2
\subseteq\operatorname{Supp}\Gamma_1.
\]
\end{theorem}

\begin{proof}
Let $\mu=\#\{j:p\mid U_j\}$. The definition of $\Gamma_m$ takes the gcd
of all products supported on $d-m+1$ indices. If $\mu\ge m$, the complement
of the divisible-index set has size at most $d-m$; consequently no set of
$d-m+1$ indices can avoid all divisible numerators. Every product entering the
gcd is therefore divisible by $p$, and $p\mid\Gamma_m$.

Conversely, if $\mu<m$, then at least $d-m+1$ of the $U_j$ are prime to $p$.
Choose a set $J$ consisting only of such indices. The corresponding product in
\eqref{eq:Gamma-m} is prime to $p$, so the gcd is also prime to $p$.
Theorem~\ref{thm:cramer-numerator} identifies $\mu$ with the number of zero
entries in the unique congruence vector, giving \eqref{eq:Gamma-layer}.
\end{proof}

\begin{definition}[Cramer polynomial]\label{def:cramer-poly}
Define
\begin{equation}\label{eq:cramer-poly}
\mathscr C_r(X)=\prod_{j=1}^d\bigl(V_j(r)X-U_j(r)\bigr)\in\mathbf Z[X].
\end{equation}
\end{definition}

\begin{corollary}[Multiplicity form of simultaneous vanishing]\label{cor:poly-zeros}
For every all-weight good prime $p$,
\begin{equation}\label{eq:poly-zeros}
\boxed{\;
\#\{j:D_j=0\}
=
\operatorname{ord}_{X=0}\bigl(\mathscr C_r(X)\bmod p\bigr).
\;}
\end{equation}
\end{corollary}

\begin{proof}
Every $V_j$ is a unit modulo a good prime. The factor $V_jX-U_j$ is divisible
by $X$ modulo $p$ exactly when $p\mid U_j$. The multiplicities add, and
Theorem~\ref{thm:cramer-numerator} completes the proof.
\end{proof}

Writing
\[
\mathscr C_r(X)=c_0+c_1X+\cdots+c_dX^d,
\]
Corollary~\ref{cor:poly-zeros} may equivalently be stated as
\[
\#\{j:D_j=0\}\ge m
\quad\Longleftrightarrow\quad
c_0\equiv c_1\equiv\cdots\equiv c_{m-1}\equiv0\pmod p.
\]
Thus the same simultaneous-vanishing information can be read either from the
layer integers $\Gamma_m$ or from the order of vanishing of one polynomial.
The former isolates prime support in integers, whereas the latter packages all
layers into a single object.

\section{Sparse Ramanujan-type specializations}
\label{sec:newcongruences}

For almost every prime in the safe range, Section~\ref{sec:uniquecong} already
produces a canonical reduction of the characteristic-zero identity. The examples
below isolate the additional arithmetic phenomenon: one or more Cramer
coordinates vanish modulo $p$, so the unique canonical formula becomes strictly
sparser. It is these vanishing-coordinate specializations that we call
Ramanujan-type congruences.

Once the residue vector $(Z_1,\ldots,Z_d)\bmod p$ and the nonzero residue of
$\det N_r$ are known, the congruence for every $n$ follows from the modular-form
identity; one does not need to verify infinitely many values of $n$. Coordinates
which vanish modulo $p$ disappear from the canonical formula.

For even $a,b\ge4$ put
\begin{align}
\mathfrak h_{a,b}(n)
={}&-\frac{B_a}{2}\,n\sigma_{b-1}(n)
+\frac{B_b}{2}\,n\sigma_{a-1}(n)\notag\\
&+\sum_{m=1}^{n-1}
\bigl(an-(a+b)m\bigr)
\sigma_{a-1}(m)\sigma_{b-1}(n-m).
\label{eq:hab}
\end{align}
Thus $\mathfrak h_{a_j,b_j}=h_{r,j}$.

\begin{theorem}[Four exact canonical congruences]\label{thm:new-congruences}
For every $n\ge1$,
\begin{align}
C_7(n)
&\equiv 3\tau(n)+27\,\mathfrak h_{22,4}(n)\pmod{29},
\label{eq:new-r7}\\
C_8(n)
&\equiv 16\tau(n)+7\,\mathfrak h_{24,4}(n)\pmod{41},
\label{eq:new-r8}\\
C_{11}(n)
&\equiv48\tau(n)+52\,\mathfrak h_{30,4}(n)
+48\,\mathfrak h_{26,8}(n)\pmod{179},
\label{eq:new-r11}\\
C_{28}(n)
&\equiv172\tau(n)+93\,\mathfrak h_{64,4}(n)
+116\,\mathfrak h_{60,8}(n)
+171\,\mathfrak h_{56,12}(n)\pmod{227}.
\label{eq:new-r28}
\end{align}
In the four cases the vanishing coordinate sets are respectively
\[
\{2\},\qquad \{2\},\qquad \{2\},\qquad \{2,4\}.
\]
Thus the displayed formulas are obtained from the unique canonical congruence
by omitting precisely those zero-coordinate blocks.
\end{theorem}

\begin{proof}
Exact reduction of the Cramer systems gives
\[
\begin{array}{c|c|c|c}
r&p&(Z_1,\ldots,Z_d)\bmod p&\det N_r\bmod p\\ \hline
7&29&(6,0)&17\\
8&41&(35,0)&36\\
11&179&(10,0,23)&147\\
28&227&(108,0,164,0,140)&29
\end{array}
\]
so all four primes are good. The table is obtained by a finite exact
calculation over $\mathbf Q$: one forms $h_{r,j}(n)$ for
$1\le n,j\le d$ from \eqref{eq:h}, forms $z_r$ from \eqref{eq:Nz}
(or, equivalently, uses the recurrence \eqref{eq:taurec} for the finitely many
values of $\tau(n)$ that occur), computes the Cramer determinants in
$\mathbf Q$, and only then reduces the resulting rational numbers modulo $p$.
Thus every entry in the table is an independently checkable certificate involving
only finitely many exact integer and rational operations; no floating-point
approximation and no verification over a finite sample of target indices $n$ is
used. Reducing \eqref{eq:unique-congruence} then gives the displayed coefficients
for every $n\ge1$ simultaneously.
\end{proof}

In particular, the table in the proof is already a finite exact certificate of the
four infinite families of congruences: the determinant residues certify the basis,
and the Cramer residues certify the unique coordinates. The expanded formulas
below merely translate the compact Rankin--Cohen blocks into ordinary divisor
sums and divisor--divisor convolutions.

For reference, the first three congruences expand to
\begin{align}
C_7(n)\equiv{}&
3\tau(n)+25n\sigma_3(n)+n\sigma_{21}(n)\notag\\
&+27\sum_{m=1}^{n-1}(22n-26m)
\sigma_{21}(m)\sigma_3(n-m)\pmod{29},
\label{eq:new-r7-expanded}\\
C_8(n)\equiv{}&
16\tau(n)+34n\sigma_3(n)+32n\sigma_{23}(n)\notag\\
&+7\sum_{m=1}^{n-1}(24n-28m)
\sigma_{23}(m)\sigma_3(n-m)\pmod{41},
\label{eq:new-r8-expanded}\\
C_{11}(n)\equiv{}&
48\tau(n)+36n\sigma_3(n)+23n\sigma_{29}(n)
+37n\sigma_7(n)+35n\sigma_{25}(n)\notag\\
&+52\sum_{m=1}^{n-1}(30n-34m)
\sigma_{29}(m)\sigma_3(n-m)\notag\\
&+48\sum_{m=1}^{n-1}(26n-34m)
\sigma_{25}(m)\sigma_7(n-m)\pmod{179}.
\label{eq:new-r11-expanded}
\end{align}
The two-coordinate-vanishing example is
\begin{align}
C_{28}(n)\equiv{}&
172\tau(n)+100n\sigma_3(n)+146n\sigma_{63}(n)
+223n\sigma_7(n)+104n\sigma_{59}(n)\notag\\
&+192n\sigma_{11}(n)+198n\sigma_{55}(n)\notag\\
&+93\sum_{m=1}^{n-1}(64n-68m)
\sigma_{63}(m)\sigma_3(n-m)\notag\\
&+116\sum_{m=1}^{n-1}(60n-68m)
\sigma_{59}(m)\sigma_7(n-m)\notag\\
&+171\sum_{m=1}^{n-1}(56n-68m)
\sigma_{55}(m)\sigma_{11}(n-m)
\pmod{227}.
\label{eq:new-r28-expanded}
\end{align}

\section{Jacobian and one-variable realization}
\label{sec:jacobian}

The preceding sections describe the congruence coefficients through Fourier
coordinates. There is also a purely algebraic realization which removes the
$cusp$ factor $\Delta$ and places the entire coordinate problem in the polynomial
ring $\mathbf Q[E_4,E_6]$. Besides shortening computations, this explains why
the same Cramer vector reappears after passage to a one-variable model.

Put
\[
X=E_4,\qquad Y=E_6,
\]
so the graded ring of level-one modular forms is
$\mathbf Q[X,Y]$ with weights $4$ and $6$.

\begin{theorem}[Rankin--Cohen brackets as Jacobians]\label{thm:jacobian}
Let $f,g\in\mathbf Q[X,Y]$ be homogeneous modular forms of weights $a,b$.
Then
\begin{equation}\label{eq:jacobian}
\boxed{\;
\frac{[f,g]_1}{\Delta}
=
-3456\,
\det\begin{pmatrix}
\partial f/\partial X&\partial f/\partial Y\\
\partial g/\partial X&\partial g/\partial Y
\end{pmatrix}.
\;}
\end{equation}
\end{theorem}

\begin{proof}
Ramanujan's differential identities give
\[
DX=\frac{E_2X-Y}{3},\qquad
DY=\frac{E_2Y-X^2}{2}.
\]
By the chain rule,
\[
Df=f_XDX+f_YDY,\qquad Dg=g_XDX+g_YDY.
\]
Because $f$ and $g$ are weighted homogeneous of weights $a$ and $b$,
\[
4Xf_X+6Yf_Y=af,\qquad
4Xg_X+6Yg_Y=bg.
\]
Insert the chain-rule expressions into $[f,g]_1=afDg-b(Df)g$. The terms
containing $E_2$ combine to
\[
\frac{E_2}{6}
\bigl(af(2Xg_X+3Yg_Y)-bg(2Xf_X+3Yf_Y)\bigr)=0
\]
by the weighted Euler identities. The remaining terms factor as
\[
[f,g]_1=-2(X^3-Y^2)(f_Xg_Y-f_Yg_X).
\]
Finally $X^3-Y^2=1728\Delta$, which proves \eqref{eq:jacobian}.
\end{proof}

Dividing \eqref{eq:all-coordinate} by $\Delta$ gives the weight-$s$ identity
\begin{equation}\label{eq:after-delta}
-2sG_s=\sum_{j=1}^d Z_j(r)Q_j,\qquad
Q_j:=\frac{H_j}{\Delta}\in M_s.
\end{equation}
By Theorem~\ref{thm:jacobian},
\[
Q_j=-3456\,J(G_{a_j},G_{b_j}),
\]
where $J$ denotes the polynomial Jacobian in $(X,Y)$.

For any even weight $w$ with $M_w\ne0$, write uniquely
\[
w=12m_w+4\alpha_w+6\beta_w,\qquad
\alpha_w\in\{0,1,2\},\quad \beta_w\in\{0,1\},
\]
using the residue representative $14$ rather than $2$ when necessary. Put
$A_w=3m_w+\alpha_w$ and
\begin{equation}\label{eq:Tvariable}
T=\frac{Y^2}{X^3}.
\end{equation}
Every $F\in M_w(\mathbf Q)$ then has a unique representation
\begin{equation}\label{eq:Tpoly}
F=X^{A_w}Y^{\beta_w}P_F(T),
\end{equation}
with $\deg P_F\le\dim M_w-1$. Indeed the usual basis
\[
\Delta^jE_4^{3(m_w-j)+\alpha_w}E_6^{\beta_w}
\qquad(0\le j\le m_w)
\]
becomes, after extracting the common factor $X^{A_w}Y^{\beta_w}$,
a polynomial basis in
\[
1-T=\frac{X^3-Y^2}{X^3}=\frac{1728\Delta}{E_4^3}.
\]
This also proves uniqueness of $P_F$.

\begin{proposition}[One-variable Jacobian formula]\label{prop:onevar-jac}
Suppose
\[
f=X^AY^\beta P(T),\qquad
g=X^CY^\delta Q(T)
\]
have weights $a$ and $b$. Then
\begin{equation}\label{eq:Psi}
J(f,g)
=
X^{A+C-1}Y^{\beta+\delta-1}
\left[
(A\delta-\beta C)PQ
+\frac a2\,T\,PQ'
-\frac b2\,T\,P'Q
\right].
\end{equation}
After the forced monomial cancellation, the bracket quotient
$[f,g]_1/\Delta$ is therefore represented by a polynomial in $T$.
\end{proposition}

\begin{proof}
Since
\[
T_X=-\frac{3T}{X},\qquad T_Y=\frac{2T}{Y},
\]
one has
\[
f_X=X^{A-1}Y^\beta(AP-3TP'),\qquad
f_Y=X^AY^{\beta-1}(\beta P+2TP'),
\]
and similarly
\[
g_X=X^{C-1}Y^\delta(CQ-3TQ'),\qquad
 g_Y=X^CY^{\delta-1}(\delta Q+2TQ').
\]
Factoring $X^{A+C-1}Y^{\beta+\delta-1}$ from the determinant leaves four
terms. The two $P'Q'$ contributions cancel, while the coefficients of $PQ'$ and
$P'Q$ simplify using
\[
2A+3\beta=\frac a2,\qquad
2C+3\delta=\frac b2.
\]
This gives \eqref{eq:Psi}. If $\beta+\delta=0$, the displayed prefactor has
$Y^{-1}$, but then the constant term $(A\delta-\beta C)PQ$ vanishes and every
remaining term contains $T=Y^2/X^3$; hence the apparent negative power cancels.
Thus the Jacobian is again a genuine polynomial modular form, as required.
\end{proof}

Let $R_j(T)$ denote the polynomial of $Q_j$ in the common weight-$s$
representation \eqref{eq:Tpoly}, and let $P_s(T)$ denote that of $G_s$.
Write the coefficient vectors in the basis
\[
1,\ (T-1),\ldots,(T-1)^{d-1},
\]
let $\mathcal C_r$ be the resulting $d\times d$ matrix with columns $R_j$,
and let $e_r$ be the coefficient vector of $-2sP_s$.

\begin{theorem}[Polynomial Cramer model]\label{thm:poly-cramer}
One has
\begin{equation}\label{eq:poly-system}
\mathcal C_r Z=e_r,
\qquad
Z_j(r)=
\frac{\det\mathcal C_r[j\leftarrow e_r]}{\det\mathcal C_r}.
\end{equation}
Moreover
\begin{equation}\label{eq:det-change}
\det N_r=(-1728)^{\binom d2}\det\mathcal C_r.
\end{equation}
In particular all primes $p>3$ see exactly the same Cramer ratios and the
same basis degeneracy in the Fourier and polynomial models.
\end{theorem}

\begin{proof}
Equation~\eqref{eq:poly-system} is simply
\eqref{eq:after-delta} after the common factor
$X^{A_s}Y^{\beta_s}$ has been removed. Thus the polynomial model changes the
coordinate realization but not the vector $Z$.

Next,
\[
T=\frac{E_6^2}{E_4^3}=1-1728q+O(q^2),
\]
so for every $m\ge0$,
\[
(T-1)^m=(-1728)^mq^m+O(q^{m+1}).
\]
The common factor $X^{A_s}Y^{\beta_s}$ has constant term $1$. Therefore the
matrix carrying coefficients in the ordered basis
$1,T-1,\ldots,(T-1)^{d-1}$ to the first $d$ Fourier coefficients is lower
triangular with diagonal
\[
1,-1728,(-1728)^2,\ldots,(-1728)^{d-1}.
\]
Its determinant is $(-1728)^{\binom d2}$.

Finally, multiplication by $\Delta=q+O(q^2)$ sends the first $d$ coefficients
of $Q_j$ to the first $d$ positive coefficients of $H_j=\Delta Q_j$ through a
unit lower-triangular transformation. Hence this second step has determinant
$1$. Composing the two transformations gives \eqref{eq:det-change}. Since the
factor $1728=2^6 3^3$ is a unit at every prime $p>3$, the Fourier and polynomial
models have exactly the same good-prime determinant condition and the same
Cramer ratios there.
\end{proof}

\begin{remark}[Two equivalent coordinate pictures]
The Fourier model is best suited to writing the congruences for $C_r(n)$, while
the polynomial model isolates the algebraic structure of the canonical basis.
Theorem~\ref{thm:poly-cramer} shows that these are not two different
constructions: away from $2$ and $3$ they are related by an invertible triangular
change of integral coordinates. In particular, the uniqueness theorem of
Section~\ref{sec:uniquecong} may be proved in either model.
\end{remark}


\begin{thebibliography}{99}

\bibitem{CPZ}
Y.~Choie, Y.~K.~Park and D.~Zagier,
\emph{Periods of modular forms on $\Gamma_0(N)$ and products of Jacobi theta
functions}, J. Eur. Math. Soc. \textbf{21} (2019), no.~5, 1379--1410.
\href{https://doi.org/10.4171/JEMS/864}{doi:10.4171/JEMS/864}.

\bibitem{Fukuhara}
S.~Fukuhara, \emph{Explicit formulas for Hecke operators on cusp forms, Dedekind
symbols and period polynomials}, J. Reine Angew. Math. \textbf{607} (2007),
163--216.

\bibitem{LNX}
A.~Lei, T.~Ni and H.~Xue, \emph{Linear independence of even periods of modular
forms}, J. Number Theory \textbf{248} (2023), 120--139.

\bibitem{NSSXY}
T.~Ni, A.~Song, Y.~Su, H.~Xue and A.~Yin,
\emph{Explicit generators of the space of modular forms}, Ramanujan J.
\textbf{70} (2026), Art.~47.
\href{https://doi.org/10.1007/s11139-026-01414-5}{doi:10.1007/s11139-026-01414-5}.

\bibitem{RS}
B.~Ramakrishnan and B.~Sahu, \emph{Identities for the Ramanujan tau function and
certain convolution sum identities for the divisor functions}, in Number Theory,
RMS Lecture Notes Series, no.~23, 2016, 63--75.

\bibitem{Xue}
H.~Xue, \emph{Rankin--Cohen brackets of Eisenstein series}, Int. J. Number Theory
\textbf{19} (2023), no.~7, 1563--1570.

\bibitem{XueMixed}
H.~Xue, \emph{Linear independence between odd and even periods of modular forms},
New York J. Math. \textbf{30} (2024), 521--531.

\bibitem{Zagier}
D.~Zagier, \emph{Elliptic modular forms and their applications}, in The 1-2-3 of
Modular Forms, Springer, 2008, 1--103.

\end{thebibliography}
\end{document}